\documentclass[a4paper, 10pt, twoside, notitlepage]{amsart}
\usepackage{amsmath,amscd}
\usepackage{amssymb}
\usepackage{amsthm}
\usepackage{comment}
\usepackage{graphicx, xcolor}

\usepackage{mathrsfs}
\usepackage[ocgcolorlinks, linkcolor=blue]{hyperref}

\usepackage{bm}
\usepackage{bbm}
\usepackage{url}

\usepackage[utf8]{inputenc}
\usepackage{mathtools,amssymb}
\usepackage{esint}
\usepackage{tikz}
\usepackage{dsfont}
\usepackage{relsize}
\usepackage{url}
\usepackage{xcolor}
\usepackage{graphicx}
\usepackage{mathrsfs}
\usepackage[shortlabels]{enumitem}
\usepackage{lineno}
\usepackage{amsmath}
\usepackage{enumitem}
\usepackage{amsthm} 
\usepackage{verbatim}
\usepackage{dsfont}
\numberwithin{equation}{section}

\newcommand{\s}{\hspace{0.5pt}}

\newcommand{\ccdot}{\,\cdot\,}

\newcommand{\wt}{\widetilde}

\allowdisplaybreaks

\graphicspath{{images/}}

\newtheorem{theorem}{Theorem}[section]
\newtheorem{lemma}[theorem]{Lemma}
\newtheorem*{lemma*}{Lemma}

\newtheorem{remark}[theorem]{Remark}

\title[An inverse source problem for a quasilinear elliptic equation]{An inverse source problem for a quasilinear elliptic equation}

\author[T. Liimatainen]{Tony Liimatainen}
\address{Department of Mathematics and Statistics, University of Jyv\"askyl\"a, Jyv\"askyl\"a, Finland}
\curraddr{}
\email{tony.t.liimatainen@jyu.fi}

\author[S. Jaiswal]{Shubham Jaiswal}
\address{Department of Mathematics and Statistics, University of Jyv\"askyl\"a, Jyv\"askyl\"a, Finland}
\curraddr{}
\email{shubham.s.jaiswal@jyu.fi}


\newcommand{\C}{{\mathbb C}}
\newcommand{\R}{{\mathbb R}}

\newcommand{\eps}{\epsilon}

\newcommand {\p} {\partial}

\newcommand{\LC}{\left(}
\newcommand{\RC}{\right)}

\newcommand{\abs}[1]{\left\lvert #1 \right\rvert}

\begin{document}
\maketitle

\section*{Abstract}
We initiate the study of inverse source problems for quasilinear elliptic equations of the form
\[
\left\{
\begin{array}{ll}
\nabla \cdot (\gamma(x,u,\nabla u) \nabla u) = F & \text{in } \Omega, \\
u = f & \text{on } \partial\Omega,
\end{array}
\right.
\]
where $\Omega \subset \mathbb{R}^n$, $n \geq 2$, is a simply connected bounded domain. We consider the specific nonlinearity $\gamma(x,u,\nabla u) = \sigma(x) + q(x) u$, with $q$ assumed to be known. By exploiting the nonlinearity to break the gauge invariance of the  problem, we establish unique recovery of both $\sigma$ and $F$ from the associated Dirichlet-to-Neumann (DN) map under the structural conditions $q$ and $\nabla(\sigma/q)$ are nowhere vanishing in $\overline\Omega$. In the absence of these conditions, in particular in the linear case, we demonstrate that the inverse problem admits a gauge obstructing the uniqueness.

We use higher order linearizations to obtain a complicated coupled system for the unknowns. The complexity of this system arises in part from the gauge freedom of the linearized equation, which is new in this context. We solve the system by constructing suitable complex geometric optics solutions and applying the unique continuation principle for nonlinear elliptic systems. We anticipate that the solution method developed here will prove useful in other inverse problems as well.

{\color{black}
\section{Introduction}
	
	Let $\Omega \subset\R^n$ be a simply connected bounded domain with $C^\infty$-smooth boundary $\p \Omega$ with $n\geq 2$. In this paper we consider quasilinear elliptic equations of the form
	\begin{equation}\label{eq:quasilin_equation}
\begin{cases}
\begin{array}{ll}
     \nabla \cdot ((\sigma + qu) \nabla u)=F & in\ \Omega \\
      u=f & on\ \partial \Omega 
    
\end{array}
\end{cases}
\end{equation}
	for $\alpha \in (0, 1)$, $ \sigma \in C^{2,\alpha}(\overline{\Omega})$, $q\in C^{2,\alpha}(\overline\Omega)$.  
    
Let us assume for now that the boundary value problem \eqref{eq:quasilin_equation} is well-posed on an open subset $\mathcal N \subset C^{2,\alpha}(\partial \Omega)$. In this case, the Dirichlet-to-Neumann map (DN map) is defined by the usual assignment
\begin{equation}
\Lambda_{\sigma,F} : \mathcal N
 \to C^{1,\alpha}(\partial \Omega), \quad f \mapsto \partial_\nu u_f|_{\partial \Omega},
\end{equation}
where $\nu$ denotes the unit outer normal on $\partial \Omega$. In Theorem \ref{Well-posedness} we show that if there is $f_0 \in C^{2,\alpha}(\partial \Omega)$ such that the equation \eqref{eq:quasilin_equation} admits a solution $u_0 \in C^{2,\alpha}(\Omega)$ with $u_0|_{\partial \Omega} = f_0$, and

\begin{equation}\label{eq:assumption_for_coefficients}
 \sigma + qu \geq c>0 ,\\\ x \in \overline{\Omega}\ \ \ \ \ \ \ \
\text{and} \ \ \ \ \ \ \ \
 \nabla \cdot (q \nabla u) \leq0 ,\ \ \  x \in \overline{\Omega}.
\end{equation}
\\
then there is an open neighborhood $\mathcal N \subset C^{2,\alpha}(\partial \Omega)$ of $f_0$ where \eqref{eq:quasilin_equation} is well-posed in the following sense: For each $f \in \mathcal N$ there exists a solution $u_f$ to \eqref{eq:quasilin_equation} with $u_f|_{\partial \Omega} = f$ and the solution $u_f$ is unique in a fixed neighborhood of $u_0 \in C^{2,\alpha}(\Omega)$.


Consider the equation \eqref{eq:quasilin_equation} for two sets $(\sigma,F)$ and $(\tilde\sigma, \widetilde F)$ of coefficients. Let $\Lambda=\Lambda_{\sigma,F}$ and $\widetilde \Lambda=\Lambda_{\tilde \sigma,\tilde F}$ be the corresponding DN maps defined on $\mathcal{N}_1\subset C^{2,\alpha}(\p \Omega)$ and $\mathcal{N}_2\subset C^{2,\alpha}(\p \Omega)$, respectively. When we write
\[
 \Lambda(f)=\widetilde\Lambda(f) \text{ for any } f\in \mathcal{N},
\]
we especially assume that $\mathcal N\subset \mathcal N_1\cap \mathcal N_2$. 
   
   \smallskip
   
 \begin{itemize}
 	\item   \textbf{Inverse source problem:} Can one uniquely recover 
$\sigma$ and 
$F$ from the Dirichlet-to-Neumann map 
$\Lambda_{\sigma,F}$?
 \end{itemize}

Equations of the form \eqref{eq:quasilin_equation} arise naturally in a variety of physical contexts where the effective coefficient depends linearly on the field variable itself. In heat transfer, for instance, many materials exhibit temperature dependent thermal conductivity, leading to a quasilinear steady state heat equation with internal heat sources $F$. Similarly, in subsurface hydrology, permeability can vary with fluid pressure, and in nonlinear electrostatics, conductivity may depend on the electric potential. Such models also appear in biological tissues where diffusion coefficients depend on concentration.  In electrical impedance tomography (EIT)  the simplest nonlinear corrections to linear conductivity models are captured by this structure. 

In all these applications, the inverse problem of recovering both the material properties and the internal sources from boundary measurements is of paramount importance. However, as demonstrated below for the linear case, such inverse source problems typically suffer from a gauge invariance that prevents unique recovery, motivating the investigation of whether the nonlinearity can be exploited to break this gauge and restore uniqueness.


	\begin{remark}\label{remark:counterexample} 
	Let us consider the inverse source problem for the linear equation
	
	\begin{align}\label{l1}
		\begin{cases}
			\nabla\cdot(\sigma \nabla u)=F &\text{ in }\Omega,\\
			u=f& \text{ on }\p \Omega.
		\end{cases}
	\end{align}
	The question in this inverse problem is whether the Dirichlet-to-Neumann map
\[
\Lambda_F : C^{\infty}(\partial\Omega) \to C^{\infty}(\partial\Omega)
\]
uniquely determines the function $F$.
  In general, the answer is negative for the reason outlined below. Let $u$ solve \eqref{l1} and let $\varphi$ be an arbitrary $C^2$-function satisfying $\left. \varphi \right|_{\p \Omega} =\left. \p_\nu \varphi \right|_{\p \Omega} =0$ and define
     \begin{align}\label{eq_r1}
     	\tilde u:=u+\varphi.
     \end{align}
	Therefore, we have $\LC \tilde u|_{\p \Omega}, \p_\nu \tilde u|_{\p \Omega}\RC=\LC u|_{\p \Omega}, \p_\nu u|_{\p \Omega}\RC$, and 
	\begin{align}\label{eq_r2}
		\begin{split}
			\nabla\cdot(\sigma \nabla \tilde u) =&\nabla\cdot(\sigma \nabla (u + \varphi))\\
            =&\nabla\cdot(\sigma \nabla u) + \nabla\cdot(\sigma \nabla \varphi)\\
			=&F+\nabla\cdot(\sigma \nabla \varphi)=\widetilde F.
		\end{split}
	\end{align}
	Hence, the functions $u$ and $\tilde u$ satisfy
\[
\nabla\cdot(\sigma \nabla u) = F 
\quad \text{and} \quad
\nabla\cdot(\sigma \nabla \tilde u) = \widetilde F,
\]
respectively. Since $u$ and $\tilde u$ have identical Cauchy data on $\partial\Omega$, 
the Dirichlet-to-Neumann maps coincide, that is,
\[
\Lambda_{F}(f) = \Lambda_{\widetilde F}(f)
\quad \text{on } \partial\Omega.
\]
It follows that the source term cannot, in general, be uniquely recovered from the Dirichlet-to-Neumann map. 
\end{remark}

As was observed originally in \cite{liimatainenLin2024} for an inverse source problem for a semilinear equation the gauge of the problem sometimes breaks, depending on the form of the nonlinearity. In this work we prove the following uniquenes result. 

\begin{theorem}\label{thm:Main theorem}
Let $\Omega \subset \mathbb{R}^n$ be a bounded and simply connected domain with $C^\infty$-smooth
boundary $\partial\Omega$, $n \ge 2$. Let 
$\sigma,\tilde\sigma, q,\tilde q, F,\widetilde F \in C^{\infty}(\overline{\Omega})$ and $f_0 \in C^{\infty}(\partial \Omega)$.  Assume the structural conditions \[q(x)\neq 0 \  \text{ and } \  \nabla(\sigma/q)(x) \neq 0, \quad \text{ for all } x  \in \overline{\Omega} \]
and further that 
\[
q=\tilde q  \text{ in } \overline{\Omega}
\]

Assume that there exists an open set $\mathcal{N} \subset C^{2,\alpha}(\partial\Omega)$
such that the Dirichlet-to-Neumann maps $\Lambda_{\sigma,F}$ and $\Lambda_{\tilde\sigma,\widetilde F}$ of the boundary value problems
\begin{equation}
\begin{cases}
\begin{array}{ll}
     \nabla \cdot ((\sigma + qu) \nabla u)=F & in\ \Omega \\
       u=f & on\ \partial \Omega 
\end{array} 
\end{cases}
\end{equation} and
\begin{equation}
\begin{cases}
\begin{array}{ll}
     \nabla \cdot ((\tilde\sigma + q\tilde u) \nabla \tilde u)=\widetilde F & in\ \Omega \\
       \tilde u=f & on\ \partial \Omega 
\end{array} 
\end{cases}
\end{equation}
satisfy
\[
\Lambda_{\sigma,F} = \Lambda_{\tilde\sigma,\widetilde F}
\quad \text{for all } f \in \mathcal{N}.
\] 
Then 
\[
\sigma= \tilde\sigma\quad \text{and} \quad F =\widetilde F \quad \text{in } \Omega.
\]
\end{theorem} 
We record the following remarks regarding the assumptions in the theorem.

\begin{remark}\label{rem:q-nonzero-assumption}
$\bullet$ The assumption that \(q\neq 0\) $\forall x \in \overline{\Omega}$ ensures that the equation is genuinely nonlinear and is essential for our analysis. If \(q\) were to vanish on an open set, the equation would become linear there, and would lead to non-uniqueness in the inverse problem as illustrated in Remark \ref{remark:counterexample}.

$\bullet$  The other structural condition
\[
\nabla(\sigma/q)(x) \neq 0, \quad \text{ for all } x \in \overline{\Omega}
\]
 is also necessary in the general setting where also \(q\) is unknown. Indeed, if this condition was violated, one could multiply the equation
\[
\nabla\cdot\big((\sigma+qu)\nabla u\big)=F
\]
by the constant \(\sigma/q\) to obtain
\[
\nabla\cdot\big((\sigma^2/q+\sigma u)\nabla u\big)=\sigma F/q.
\]
This would imply that the coefficients \((\sigma,q,F)\) and \((\sigma^2/q,\sigma,\sigma F/q)\) yield the same Dirichlet-to-Neumann map, thereby obstructing unique recovery.

$\bullet$  The additional assumption \(q = \tilde q\) is technical and arises from the complexity of the proof. We expect that this restriction could be removed through a more refined analysis of the highly coupled system for the unknowns that appears in the proof, and by possibly considering higher-order linearizations  of the Dirichlet-to-Neumann map. However, as the analysis in the present case is already quite involved, we leave the treatment of the  general case to a future work.

$\bullet$  The assumption that $\Omega$ is simply connected is required to integrate the gauge condition arising from the linearized equation. The regularity assumptions on the coefficients can be reduced in the proof to finite regularity, though not quite to the minimal regularity required for the well-posedness result in Theorem \ref{Well-posedness}.

\end{remark}

\subsection{Earlier works}
The standard approach to inverse problems for nonlinear elliptic equations was initiated in \cite{isakov1993uniqueness_parabolic}, where the author linearized the nonlinear Dirichlet-to-Neumann map. This linearization reduced the inverse problem for a nonlinear equation to that of a linear equation, which could then be addressed using techniques for linear problems.

Subsequently, second-order linearizations, where the data depends on two independent parameters, were employed to solve inverse problems in works such as \cite{AYT2017direct,CNV2019reconstruction,KN002,sun1996quasilinear,sun2010inverse,sun1997inverse}.

For the case $F = 0$ in \eqref{eq:quasilin_equation} (so that $u \equiv 0$ is a solution), inverse problems for semilinear elliptic equations were recently studied in \cite{FO19,LLLS2019nonlinear}. A key novelty of these works is that nonlinearity is no longer viewed as a complication but rather as a beneficial tool. This approach originates from the seminal work \cite{KLU2018}, which treated inverse problems for nonlinear wave equations in Lorentzian spacetimes. By leveraging nonlinearity in this way, inverse problems for nonlinear equations have been solved in cases where the corresponding linear inverse problems remain open. This method is now commonly known as the \emph{higher-order linearization method}.

Following \cite{KLU2018,FO19,LLLS2019partial}, the literature on inverse problems for nonlinear equations based on the higher-order linearization method has grown substantially. Without aiming to be exhaustive, we mention that works such as \cite{LLLS2019partial,KU2019remark,KU2019partial,FLL2021inverse} have investigated inverse problems for semilinear elliptic equations with general nonlinearities, including in the case of partial data.
Furthermore, a survey of inverse problems for nonlinear partial differential equations is provided in \cite{L25}.

Inverse problems for quasilinear equations using higher-order linearization have been studied in \cite{KKU2022partial,CFKKU2021calderon,FKU2021inverse,LW23}. Meanwhile, \cite{CLLO2022inverse,Nur23,Nur24,ABN20,CLT24,Mu24} considered inverse problems for minimal surface equations and harmonic map equations, which are quasilinear, on Riemannian surfaces and Euclidean domains. In \cite{CJ2025}, the authors studied the Calderón problem for quasilinear conductivities on CTA manifolds.

In \cite{liimatainenLin2024}, it was shown that for an inverse source problem (where $F \neq 0$) for a semilinear equation, the gauge invariance of the problem can sometimes be broken. The breaking of gauge invariance in inverse source problems for reaction–diffusion equations and nonlinear wave equations was further investigated in \cite{KLL24,QXYZ2025,LLPT25}. An inverse source problem for the fully nonlinear Monge–Ampère equation was studied in \cite{liimatainen2025inverse}, while an inverse problem for the prescribed mean curvature equation were addressed in \cite{LN26}.

 \subsection{Method of proof}
The proof relies on recovering a coupled system of equations for the unknowns from linearizations of the Dirichlet-to-Neumann (DN) map. The resulting system consists of three elliptic equations and equations whose coeffiients depend on solutions to the adjoint of the linearized equation, all coupled to one another. We now explain how this system is obtained and solved.

Let $(\sigma, F)$ and $(\tilde\sigma, \widetilde F)$ be coefficients for \eqref{eq:quasilin_equation} and let $u$ and $\tilde u$ denote the corresponding solutions, with $u_0$ and $\tilde u_0$ corresponding to a boundary value $f_0$.

First, we linearize \eqref{eq:quasilin_equation}. After expressing it as a magnetic Schr\"odinger equation, it becomes a Laplace equation with magnetic potential $A$ and zeroth-order term $Q$. Using the knowledge of the DN map of the linearization together with known results for the Schr\"odinger equation, we recover the coefficients of the linearized equation up to a gauge function $\tilde{\varphi}$. More precisely, we obtain
\[
A - \widetilde{A} = \nabla \tilde{\varphi} \quad \text{and} \quad Q = \widetilde Q,
\]
where
\[
A = \frac{i}{2}\frac{\nabla(\sigma+qu_0)+q\nabla u_0}{\sigma+qu_0},
\]
and
\[
Q = \frac{\Delta \sigma}{2(\sigma+qu_0)} + \text{lower-order terms in } \sigma \text{ and } u_0,
\]
and similarly for $\widetilde A$ and $\widetilde Q$.

Next, we recover the gauge function $\tilde{\varphi}$ by performing a second order linearization, which yields the integral identity
\[
0 = \int_\Omega V_0\big(\nabla \cdot (q V_1 \nabla V_2) + \nabla \cdot (q V_2 \nabla V_1)\big)
- \int_\Omega \widetilde{V}_0\big(\nabla \cdot (q \widetilde{V}_1 \nabla \widetilde{V}_2) + \nabla \cdot (q \widetilde{V}_2 \nabla \widetilde{V}_1)\big),
\]
where $V_1, V_2$ are solutions to the linearized equation corresponding to $(\sigma, F)$, $V_0$ is a solution to its adjoint equation, and $\widetilde{V}_1, \widetilde{V}_2, \widetilde{V}_0$ are the analogous quantities for $(\tilde\sigma, \widetilde F)$. We show that $V_0$ and $\widetilde{V}_0$ are related by the gauge transformation $\tilde{\varphi}$, namely
\[
\widetilde{V}_0 = \frac{\sigma + q u_0}{\tilde{\sigma} + q \tilde{u}_0} e^{-i\tilde{\varphi}} V_0.
\]
Choosing $V_1$ and $V_2$ to be complex geometric optics (CGO) solutions, we obtain the identity
\begin{equation*}
\nabla \cdot (q \nabla V_0) - A \, \nabla \cdot \big(q \nabla (B V_0)\big) = 0,
\end{equation*}
where
\[
A = e^{2i\tilde \varphi} \quad \text{and} \quad B = \frac{\sigma + qu_0}{\tilde\sigma + q\tilde u_0} e^{-i\tilde\varphi}.
\]
Here $V_0$ is any solution to the adjoint equation. 

By a specific choice of $V_0$, we first recover the elliptic PDE
\[
\nabla\cdot (q \nabla B) = 0
\]
for $B$. After solving it, we obtain the equation
\[
(1-A)\nabla\cdot (q \nabla V_0) = 0
\]
for $A$. To solve for $A$, we fix a point $x_0 \in \Omega$ and construct a CGO solution $V_0$ to the adjoint equation depending on $x_0$ such that $\nabla\cdot (q \nabla V_0)(x_0) \neq 0$. This construction relies on the structural condition $\nabla(\sigma/q) \neq 0, \quad \forall x \in \overline{\Omega}$. Varying $x_0$ over $\Omega$, we conclude $A \equiv 1$, and hence $\tilde\varphi \equiv 0$.

Let $\Theta = \sigma + q u_0$ and $\widetilde{\Theta} = \tilde\sigma + q \tilde u_0$. With $\tilde\varphi = 0$, taking the divergence of $A$ yields a quasilinear coupled elliptic system for $\hat\Theta := \Theta - \widetilde{\Theta}$ and $\hat u_0 := u_0 - \tilde u_0$:
\[
\begin{pmatrix}
\frac{1}{\Theta} & \frac{q}{\Theta} \\[4pt]
\frac{1}{2\Theta} & -\frac{q}{2\Theta}
\end{pmatrix}
\begin{pmatrix}
\Delta \hat{\Theta} \\[4pt]
\Delta \hat{u}_0
\end{pmatrix}
= \text{lower-order terms}.
\]
Applying unique continuation results for this elliptic system, we conclude
\[
\Theta = \widetilde{\Theta}, \qquad u_0 = \tilde u_0.
\]
These identities imply $\sigma = \tilde\sigma$ and $u_0 = \tilde u_0$. Substituting these relations into \eqref{eq:quasilin_equation} gives
\[
F = \widetilde{F},
\]
which completes the proof.

\subsection*{Acknowledgements}

T.L.\ and S.J. were partly supported by the Research Council of Finland (Centre of Excellence in Inverse Modelling and Imaging and FAME Flagship, grants 353091 and 359208). The authors
would like to thank Mikko Salo for helpful discussions.

\section{Preliminaries}

In this section, we establish local well-posedness for the Dirichlet problem \eqref{eq:quasilin_equation} in a neighborhood of a prescribed solution. Let $0 < \alpha < 1$ and $\delta > 0$, and define
\begin{equation}
\mathcal N_\delta := \{ f \in C^{2,\alpha}(\partial \Omega) : \|f\|_{C^{2,\alpha}(\partial \Omega)} \leq \delta \}.
\end{equation}
Note that if the source term $F$  in 
\[
\nabla \cdot ((\sigma + qu)\nabla u)=F
\]
 does not vanish identically, then zero is not a solution of \eqref{eq:quasilin_equation}. Consequently, the standard local theory developed in \cite{LLLS2019nonlinear,CFKKU2021calderon} cannot be applied directly. Instead, our argument follows the one  in \cite{liimatainenLin2024}.

\begin{theorem}\label{Well-posedness}
\textbf{(Well-posedness)}  
Let $\Omega \subset \mathbb{R}^n$ be a bounded domain with $C^\infty$ boundary and $n \ge 2$. Let $\alpha \in (0,1)$, $\sigma,q \in C^{1,\alpha}(\overline{\Omega})$, $F \in C^{\alpha}(\overline{\Omega})$, and $f_0 \in C^{2,\alpha}(\partial\Omega)$. 

Assume that there exists a solution $u_0 \in C^{2,\alpha}(\overline{\Omega})$ of
\begin{equation}\label{eq:non-linear_well_posedness}
\begin{cases}
\nabla \cdot ((\sigma + qu_0)\nabla u_0)=F & \text{in }\Omega,\\
u_0=f_0 & \text{on }\partial\Omega .
\end{cases}
\end{equation}
Suppose further that
\begin{equation}\label{eq:well-posedness_conditions}
\sigma + qu_0 \ge c>0 \quad \text{in } \overline{\Omega},
\qquad
\nabla\cdot(q\nabla u_0)\le 0 \quad \text{in } \overline{\Omega}.
\end{equation}

Then there exist constants $\delta>0$ and $C>0$ such that for every $f \in \mathcal N_\delta$ the problem
\begin{equation}
\begin{cases}
\nabla \cdot ((\sigma + qu)\nabla u)=F & \text{in }\Omega,\\
u=f_0+f & \text{on }\partial\Omega
\end{cases}
\end{equation}
admits a unique solution $u \in C^{2,\alpha}(\overline{\Omega})$ in the class
\[
\{ w \in C^{2,\alpha}(\overline{\Omega}) : \|w-u_0\|_{C^{2,\alpha}(\overline{\Omega})}\le C\}.
\]

Moreover, the solution operator and the Dirichlet-to-Neumann map
\begin{equation}\label{eq:Frechet_differentiability}
S : \mathcal N_\delta \to C^{2,\alpha}(\overline{\Omega}), 
\quad f \mapsto u,
\qquad
\Lambda : \mathcal N_\delta \to C^{1,\alpha}(\partial\Omega),
\quad f \mapsto \partial_\nu u|_{\partial\Omega},
\end{equation}
are $C^\infty$ Fréchet differentiable.
\end{theorem}
\begin{remark}
    The condition
\[
\sigma + q u_0 \geq c > 0
\]
is needed to ensure that the differential operator in \eqref{eq:quasilin_equation} is \emph{uniformly elliptic}, which allows us to apply the theory of elliptic equations to guaranty well-posedness. 

The second assumption,
\[
\nabla \cdot (q \nabla u_0) \leq 0,
\]
ensures the applicability of the \emph{maximum principle}, which is needed for the uniqueness of the solution.
\end{remark}

\begin{proof}
The proof relies on the implicit function theorem in Banach spaces.

Let 
\[
\mathcal{B}_1=C^{2,\alpha}(\partial\Omega),\qquad
\mathcal{B}_2=C^{2,\alpha}(\overline{\Omega}),\qquad
\mathcal{B}_3=C^{0,\alpha}(\overline{\Omega})\times C^{2,\alpha}(\partial\Omega).
\]

Define the map
\[
\Psi:\mathcal{B}_1\times\mathcal{B}_2\to\mathcal{B}_3,
\]
\[
\Psi(f,u)=\big(\nabla\cdot((\sigma+qu)\nabla u)-F,\; u|_{\partial\Omega}-(f_0+f)\big).
\]

Since $u_0$ solves \eqref{eq:non-linear_well_posedness}, we have
\[
\Psi(0,u_0)=(0,0).
\]

 Next, we compute the linearization with respect to $u$ at $(0,u_0)$. For $v\in\mathcal{B}_2$,
\[
D_u\Psi|_{(0,u_0)}(v)
=
\left(
\nabla\cdot((\sigma+qu_0)\nabla v)
+q\nabla u_0\cdot\nabla v
+\nabla\cdot(q\nabla u_0)v,
\;
v|_{\partial\Omega}
\right).
\]

The differential operator appearing in the first component is a uniformly elliptic second-order operator due to the positivity condition $\sigma+qu_0\ge c>0$. Together with assumption $\nabla\cdot(q\nabla u_0)\le0$, the maximum principle applies (see \cite[Corollary 3.2]{gilbarg2015elliptic}). 

Standard Schauder theory for linear elliptic boundary value problems (see \cite[Theorems 6.6 and 6.8]{gilbarg2015elliptic}) implies that for any
\[
(h_1,h_2)\in C^{0,\alpha}(\overline{\Omega})\times C^{2,\alpha}(\partial\Omega)
\]
there exists a unique solution $v\in C^{2,\alpha}(\overline{\Omega})$ to
\[
\begin{cases}
\nabla\cdot((\sigma+qu_0)\nabla v)
+q\nabla u_0\cdot\nabla v
+\nabla\cdot(q\nabla u_0)v
= h_1 & \text{in }\Omega,\\
v=h_2 & \text{on }\partial\Omega.
\end{cases}
\]

Moreover, the Schauder estimate yields
\[
\|v\|_{C^{2,\alpha}(\overline{\Omega})}
\le C\big(
\|h_1\|_{C^{0,\alpha}(\overline{\Omega})}
+\|h_2\|_{C^{2,\alpha}(\partial\Omega)}
\big),
\]
which shows that $D_u\Psi|_{(0,u_0)}:\mathcal{B}_2\to\mathcal{B}_3$ is a bounded bijection with bounded inverse.

Thus $D_u\Psi|_{(0,u_0)}$ is an isomorphism between the Banach spaces $\mathcal{B}_2$ and $\mathcal{B}_3$. Applying the implicit function theorem in Banach spaces (see \cite[Theorem 10.6, Remark 10.5]{RR06}), there exist $\delta>0$ and a $C^\infty$ mapping
\[
S:\mathcal N_\delta\to C^{2,\alpha}(\overline{\Omega})
\]
such that
\[
\Psi(f,S(f))=(0,0)
\]
for all $f\in \mathcal N_\delta$.

Consequently, $u=S(f)$ is the unique solution of the boundary value problem in a neighborhood of $u_0$. The estimate
\[
\|u-u_0\|_{C^{2,\alpha}(\overline{\Omega})}\le C\|f\|_{C^{2,\alpha}(\partial\Omega)}
\]
follows from the implicit function theorem.

Finally, since the solution operator $S$ is $C^\infty$ in the Fréchet sense and the normal derivative defines a continuous linear mapping
\[
C^{2,\alpha}(\overline{\Omega})\to C^{1,\alpha}(\partial\Omega),
\]
the Dirichlet-to-Neumann map
\[
\Lambda(f)=\partial_\nu S(f)|_{\partial\Omega}
\]
is also $C^\infty$.
\end{proof}
In the next remark we show that the class of coefficients satisfying the assumption on the existence of a solution $u_0\in C^{2,\alpha}(\overline\Omega)$ satisfying \eqref{eq:well-posedness_conditions} contains for example all non-positive sources $F$ at least when $\sigma=q=1$.

\begin{remark}
Let $\sigma=1$, $q=1$, $F \leq 0$ and consider
\begin{equation}\label{special_case}
\begin{cases}
\nabla \cdot ((1 + u_0) \nabla u_0)=F & \text{in } \Omega \\
u=f_0 & \text{on } \partial \Omega.
\end{cases}
\end{equation}
We show that this equation admits a nonnegative solution $u$ for any boundary data $f_0\geq 0$ that satisfies \eqref{eq:well-posedness_conditions} In particular, we can take $u_0$ in Theorem \ref{Well-posedness} to be the solution corresponding to any $f_0\geq 0$.

To see this, let $w$ solve the linear Dirichlet problem
\begin{equation}\label{intermediate_dirichlet}
\begin{cases}
\Delta w = F & \text{in } \Omega \\
w = f + \frac{1}{2}f^2 & \text{on } \partial \Omega.
\end{cases}
\end{equation}
We first show that $w \geq 0$ in $\Omega$. By the maximum principle, since $F \leq 0$, the function $w$ attains its minimum on the boundary. Hence
\[
\min_{\overline{\Omega}} w = \min_{\partial\Omega} w = \min_{\partial\Omega} \left(f + \frac{1}{2}f^2\right) \geq 0,
\]
where the last inequality follows from $f \geq 0$. Thus $w \geq 0$ in $\overline{\Omega}$.

Next we define $u_0:=-1+\sqrt{1+2w}$, then $w=u_0 + \frac{1}{2}u_0^2$ and we immediately see that $u_0$ solves \eqref{special_case}.
We also see  that the conditions in \eqref{eq:well-posedness_conditions} hold:
     \[\sigma + q u_0= 1+u_0=\sqrt{1+2w}\geq 1>0\]
     and
\[\nabla\cdot(q\nabla u_0)=\Delta u_0= \frac{F-\abs{\nabla u_0}^2}{1+u_0} \leq 0\] as $1+u_0 >0$.

Thus, by choosing any nonnegative boundary data $f_0$, we obtain a solution $u_0$ satisfying the conditions of Theorem \ref{Well-posedness}. 
\end{remark}

\section{Higher order linearization}
Let us consider the nonlinear equation 
\begin{equation}\label{eq:non-linear_well_posedness_eps_dependent} 
\begin{cases}
\begin{array}{ll}
     \nabla \cdot ((\sigma + qu_0) \nabla u_0)=F & in\ \Omega \\
       u=f_0+\eps_1f_1+\eps_2f_2 & on\ \partial \Omega 
\end{array} 
\end{cases}
\end{equation}
where $\eps_1,\eps_2$ are small and $f_1,f_2\in C^{2,\alpha}(\p\Omega)$.
We compute the linearizations of this up to second order in the parameters $\eps_1$ and $\eps_2$. We denote $\eps=(\eps_1,\eps_2)$. We remark that differentiating the equation with respect to parameters is justified by \eqref{eq:Frechet_differentiability}.

The first linearized equation with respect to $\eps_j$, $j=1,2$, is 
\begin{equation}\label{eq:first_linearized_equation}
\nabla\cdot (\sigma+qu_0)\nabla V_j+q\nabla u_0\cdot \nabla V_j+ \nabla\cdot (q\nabla u_0)V_j=0.
\end{equation}
Here
\[
V_j=\p_{\eps_j}|_{\eps=0} u_\eps,
\]
where $u_\eps$ solves \eqref{eq:non-linear_well_posedness_eps_dependent}. We also write the linearized equation as
\[
LV=0,
\]
where
$L$ is the linear operator 
\begin{equation}\label{eq:linearized_op}
L=\nabla\cdot (\sigma+qu_0)\nabla +q\nabla u_0\cdot \nabla + \nabla\cdot (q\nabla u_0).
\end{equation}


Taking the mixed derivative $\p^2_{\eps_1\eps_2}$ of \eqref{eq:non-linear_well_posedness_eps_dependent}, we obtain the equation
\begin{multline}\label{eq:second_linearized_equation}
    0=\frac{\p^2}{\p\eps_1\p\eps_2}\Big|_{\eps=0}\left(\nabla \cdot ((\sigma + qu_0) \nabla u_0)\right) \\
    =\frac{\p}{\p\eps_1}\Big|_{\eps=0}\left(\nabla\cdot(\sigma+qu_\eps)\nabla\p_{\eps_2}u_\eps+\nabla\cdot(q\p_{\eps_2}u_\eps\nabla u_\eps)  \right) \\
    =\nabla\cdot(\sigma+qu_0)\nabla w+\nabla \cdot(qV_1\nabla V_2)+\nabla \cdot(qV_2\nabla V_1)+ \nabla \cdot (qw)\nabla u_0.
\end{multline}
Here 
\[
w=\p^2_{\eps_1\eps_2}|_{\eps=0} u_\eps.
\]
The above is the same as 
\begin{equation}\label{eq:2nd_lin_eq}
Lw=-\nabla \cdot(qV_1\nabla V_2)-\nabla \cdot(qV_2\nabla V_1).
\end{equation}

By using \eqref{eq:linearized_op}, we have for $f$ and $g$ $\in C^{2,\alpha}(\Omega)$ complex valued that 
\begin{align*}
    \int_{\Omega} \bar f\s L(g)
    &=\int_{\Omega} 
    \bar{f} \s \left(\nabla\cdot (\sigma+qu_0)\nabla +q\nabla u_0\cdot \nabla + \nabla\cdot (q\nabla u_0)\right)g \\
    &= \int_{\Omega} \bar{f}\s (\nabla\cdot (\sigma+qu_0)\nabla g) + \int_{\Omega}\bar{f}\s \nabla\cdot(q g \nabla u_0)\\
    &= - \int_{\Omega} \nabla \bar{f} \ccdot (\sigma+qu_0)\nabla g) - \int_{\Omega}\nabla \bar{f}\ \cdot (qg\nabla u_0) \\
    &\qquad + \int_{\partial{\Omega}} \bar{f}\ (\sigma+qu_0) \nu \cdot \nabla g + \int_{\partial{\Omega}} (\bar{f} q g)\s \nu \ccdot \nabla u_0 \\
    &= \int_\Omega  (\nabla\cdot (\sigma+qu_0)\nabla \bar{f})g - \int_{\Omega}(\nabla \bar{f}\ \cdot (q\nabla u_0)) g \\
&\qquad+\int_{\partial{\Omega}} \bar{f}\ (\sigma+qu_0) \nu \cdot \nabla g 
    + \int_{\partial{\Omega}} \bar{f} q g\s \nu \ccdot \nabla u_0
    -\int_{\partial{\Omega}} g (\sigma+qu_0) \nu \cdot \nabla \bar{f}.
    \end{align*}
    Since $u_0$, $q$ and $\sigma$ are real, the above is equal to
    \begin{multline*}
         \int_\Omega  \overline{(\nabla\cdot (\sigma+qu_0)\nabla f)}g - \int_{\Omega}\overline{(\nabla f\ \cdot (q\nabla u_0))} g \\
         +\int_{\partial{\Omega}} \bar{f}\ (\sigma+qu_0) \nu \cdot \nabla g 
    + \int_{\partial{\Omega}} \bar{f} q g\s \nu \ccdot \nabla u_0
    -\int_{\partial{\Omega}} g (\sigma+qu_0) \nu \cdot \nabla \bar{f}.
         \end{multline*}
In the case the boundary values of $f$ and $g$ were zero, then the above computation shows that 
    \[
    \langle f,L g\rangle=\langle(\nabla\cdot (\sigma+qu_0)\nabla)f- (q\nabla u_0) \cdot \nabla f,g \rangle.
    \]
    This means that the Hermitian adjoint of $L$ is 
\begin{equation}\label{eq:Lstar}
    L^{*}:=(\nabla\cdot (\sigma+qu_0)\nabla )- (q\nabla u_0) \cdot \nabla .
    \end{equation}
So in general, we may write
         \begin{multline*}
        \int_{\Omega} \bar{f}\s L g= \int_\Omega (L^{*}\overline f) g  \\+\int_{\partial{\Omega}} \bar{f}\ (\sigma+qu_0) \nu \cdot \nabla g 
    + \int_{\partial{\Omega}} \bar{f} q g\s \nu \ccdot \nabla u_0
    -\int_{\partial{\Omega}} g (\sigma+qu_0) \nu \cdot \nabla \bar{f}. 
    \end{multline*}
    
Let us then derive an integral identity for the second linearization \eqref{eq:2nd_lin_eq}. For this, let $v_0$ be a solution to $L^*V_0=0$ and $V_1$ and $V_2$ solutions to $Lv=0$ in $\Omega$. Then the above with $f=\overline V_0$ and $g=w$ and \eqref{eq:2nd_lin_eq} yields that
\begin{multline}\label{eq:2nd_lin_integ_id}
    -\int_\Omega V_0(\nabla \cdot(qV_1\nabla V_2)+\nabla \cdot(qV_2\nabla V_1)) \\=\int_{\partial{\Omega}} V_0\ (\sigma+qu_0) \nu \cdot \nabla w \ + \int_{\partial{\Omega}} V_0\ q \nabla u_0 \nu \cdot \nabla w   -\int_{\partial{\Omega}} w (\sigma+qu_0) \nu \cdot \nabla V_0.
\end{multline}

If $\sigma=\tilde \sigma$ on the boundary, then the assumption $\Lambda=\widetilde \Lambda$
yields the integral identity of the second linearization  
\begin{equation}\label{eq:2nd_lin_integral_id_substracted}
0=\int_\Omega V_0(\nabla \cdot(qV_1\nabla V_2)+\nabla \cdot(qV_2\nabla V_1))-\int_\Omega \widetilde V_0(\nabla \cdot(q \tilde V_1\nabla \tilde V_2)+\nabla \cdot(q\tilde V_2\nabla \tilde V_1)),
\end{equation}
where quantities with tildes correspond to the coefficients $(\tilde \sigma, \tilde F)$ of \eqref{eq:quasilin_equation}.

\section{Solution to the first linearized problem}\label{sec:sol_to_first_lin}
We know the DN map of the first linearized equation \eqref{eq:first_linearized_equation} by differentiating the DN map of the nonlinear equation. 
That is, we know that the DN maps of the problems 
\begin{equation}\label{eq:non-linear_well_posedness_eps_dependent1} 
\begin{cases}
\begin{array}{ll}
     LV=0 & in\ \Omega \\
       V=f & on\ \partial \Omega 
\end{array} 
\end{cases}
\end{equation}
and 
\begin{equation}\label{eq:non-linear_well_posedness_eps_dependent2} 
\begin{cases}
\begin{array}{ll}
     \widetilde L\widetilde V=0 & in\ \Omega \\
       \widetilde V=f & on\ \partial \Omega 
\end{array} 
\end{cases}
\end{equation}
are the same.

By expanding the linearized equation \eqref{eq:first_linearized_equation} and diving by $-(\sigma+qu_0)$, we obtain the equation
\begin{equation}\label{eq:first_linearized_equation-2}0=-\Delta v - \left[\frac{\nabla(\sigma+qu_0) +q \nabla u_0}{\sigma+qu_0}\right ] \cdot \nabla v - \left[\frac{\nabla q \cdot \nabla u_0 +q \Delta u_0}{\sigma+qu_0}\right ] v 
\end{equation}
for the linearized solutions, which we now denoted by $v$. We use the sign convention $\Delta=\nabla\cdot \nabla$ for the Laplacian. This is an equation of the form
\begin{equation}\label{eq:gen_2nd_ord}
-\Delta v+X\cdot \nabla v+Rv=0,
\end{equation}
where
\[
X=-\frac{\nabla(\sigma+qu_0) +q \nabla u_0}{\sigma+qu_0} \  \text{ and } \ R=-\frac{\nabla q \cdot \nabla u_0 +q \Delta u_0}{\sigma+qu_0}.
\]
By the the substitutions $A=\frac {i X}{2}$ and $Q=\frac 14 \abs{X}^2-\frac 12 \nabla \cdot X +R$, where $\abs{X}^2=X\cdot X$, this becomes the magnetic Shr\"odinger equation 
\begin{equation}\label{eq:maganetic_schrodinger}
L_{A,Q}v:= -\sum_{j=1}^{n} (\partial_{x_j}+ i A_j)^2 v +Q
v=0,
\end{equation}
where 
\begin{align}\label{A-equation}
A&=-\frac i2\frac{\nabla(\sigma+ qu_0) +q \nabla u_0}{\sigma+qu_0}, 
\\
\label{Q-equation}
Q&=\frac{\Delta(\sigma+qu_0)-q\Delta u_0}{2(\sigma + qu_0)} -\frac{\nabla q \cdot \nabla u_0}{2(\sigma+qu_0)} + \frac{q^2 \abs{\nabla u_0}^2 -\abs{\nabla (\sigma+qu_0)}^2 }{4(\sigma+qu_0)^2}.
\end{align}

Thus
\[
L_{A,Q} = -\frac{1}{\sigma+qu_0} L.
\]
We will denote by $v$ solutions to $L_{A,Q}v=0$ and by $V$ solutions to $LV=0$, although these solution sets coincide. The reason for considering the operator $L_{A,Q}$ alongside $L$ is that we wish to use the convenient transformation properties of $L_{A,Q}$, the existing construction of complex geometric optics (CGO) solutions for $L_{A,Q}v=0$ and uniqueness results for the corresponding inverse problem.

Since
\[
L = -(\sigma + q u_0) L_{A,Q},
\]
its Hermitian adjoint is given by
\[
L^* = -L_{A,Q}^* \, \overline{(\sigma + q u_0)}
    = -L_{A,Q}^*(\sigma + q u_0),
\]
where we have used that $\sigma + q u_0$ is real-valued. Here $L_{A,Q}^*$ denotes the Hermitian adjoint of $L_{A,Q}$. We will denote by $v_0$ solutions to $L_{A,Q}^*v_0=0$ and by $V_0$ solutions to $L^*V_0=0$. These solution sets do not coincide, but are related by $v_0 = (\sigma+qu_0)V_0$. We have analogous formulas for the operators corresponding to the coefficients $(\tilde{\sigma}, \tilde{u}_0)$. 

Since we know that the DN maps of the linearized equations \eqref{eq:non-linear_well_posedness_eps_dependent1} and \eqref{eq:non-linear_well_posedness_eps_dependent2} agree, we have that the DN maps for the magnetic Shr\"odinger with corresponding coefficients $(A,Q)$ and $(\wt A, \wt Q)$ agree. 
It then follows from \cite[Theorem C]{nakamura1995global}  that 
\begin{equation}\label{eq:Q_recovery}
dA=d\wt{A} \text{ and } Q=\wt Q,
\end{equation}
where $d$ is the exterior derivative. Since we assume that $\Omega$ is simply connected, we there is a purely complex $\tilde \varphi\in C^\infty(\Omega)$ such that
\begin{equation}\label{eq:A_tildeA}
A-\wt A=\nabla \tilde \varphi.
\end{equation}
By \eqref{eq:A_on_bndr}, $\tilde{\varphi}$ is constant on $\partial \Omega$. By subtracting this constant, we may normalize $\tilde{\varphi}$ so that
\[
\tilde{\varphi} = 0 \quad \text{on } \partial \Omega.
\]

The gauge freedom $\nabla \tilde\varphi$ introduces a new unknown to our problem, which we will need to recover. This will be done in Section \ref{sec:recovering_phi}.

\section{Relation of linearized solution for different coefficients}
Next we derive a relation between $v$ and $\tilde v$ where $v$ solves \eqref{eq:non-linear_well_posedness_eps_dependent1} and $\tilde v$ solves \eqref{eq:non-linear_well_posedness_eps_dependent2}. It is well known that the solutions with the corresponding coefficients $(A,Q)$ and $(A +\nabla \tilde \varphi, Q)$ satisfy
\begin{equation}\label{eq:gauge_trans}
L_{A+\nabla \tilde \varphi,Q}(e^{-i\tilde \varphi}v)=e^{-i\tilde \varphi}L_{A,Q}v
\end{equation}
for any $\tilde \varphi$ smooth enough. 
For completeness we provide the details of this fact:

\begin{align*}
&L_{A+\nabla\tilde \varphi,Q}(e^{-i\tilde \varphi}v)
= - \sum_{j=1}^{n} 
   \big(\partial_{x_j}+ iA_j + i\partial_{x_j}\tilde\varphi\big)^2(e^{-i\tilde\varphi}v)
   + Q e^{-i\tilde\varphi}v \\
&= - \sum_{j=1}^{n} 
   \big(\partial_{x_j}+ iA_j + i\partial_{x_j}\tilde \varphi\big)
   \big[\,e^{-i\tilde \varphi}(\partial_{x_j}v + iA_jv)\big]
   + Q e^{-i\tilde \varphi}v \\
&= - \sum_{j=1}^{n}
   \Big\{
     \partial_{x_j}\big(e^{-i\tilde \varphi}(\partial_{x_j}v + iA_jv)\big)
     + (iA_j+i\partial_{x_j}\tilde \varphi)
       e^{-i\tilde \varphi}(\partial_{x_j}v + iA_jv)
   \Big\}
   + Q e^{-i\tilde \varphi}v \\
&= - \sum_{j=1}^{n}
   \Big\{
     -i(\partial_{x_j}\tilde \varphi)e^{-i\tilde \varphi}(\partial_{x_j}v + iA_jv)
     + e^{-i\tilde \varphi}\partial_{x_j}(\partial_{x_j}v + iA_jv)
     + (iA_j+i\partial_{x_j}\tilde\varphi)
       e^{-i\tilde \varphi}\\
       &\qquad (\partial_{x_j}v + iA_jv)
   \Big\}
   + Q e^{-i\tilde \varphi}v \\
&= - e^{-i\tilde \varphi}
   \sum_{j=1}^{n}\big(\partial_{x_j}+ iA_j\big)^2 v
   + Q e^{-i\tilde \varphi}v \\
&= e^{-i\tilde \varphi}L_{A,Q}v.
\end{align*}
This is equivalent to 
\[
L_{\widetilde A,\widetilde Q}e^{i\tilde\varphi}v=e^{i\tilde\varphi}L_{A,Q}v.
\]

By Section \ref{boundary_determination}, we have that 
\[
A=\wt A \text{ on } \p \Omega.
\]
Thus $\tilde \varphi$  is constant on the boundary and we can redefine it to be $0$ on $\p\Omega$ while still having \eqref{eq:A_tildeA}. 

Note next that if $v$ solves $L_{A,Q}v=0$ with $v=f$ on $\p \Omega$, then 
\begin{equation}
L_{\wt A, \wt Q}e^{i\tilde \varphi }v=    L_{A-\nabla \tilde \varphi, Q}e^{i\tilde\varphi }v=e^{i\tilde\varphi}L_{A,Q}v=0.
\end{equation}
Now, let $\tilde v$ solve
\[
L_{\wt A, \wt Q}\tilde v=0
\]
with $\tilde v=f$ on $\p\Omega$. Since $\tilde \varphi=0$ on $\p \Omega$, we have that $\tilde v$ and $e^{i\tilde \varphi}v$ satisfy the same equation $L_{\wt A, \wt Q}u=0$. Moreover, by the second condition in \eqref{eq:assumption_for_coefficients}, we have that $L_{\wt A, \wt Q}u=0$ satisfy the maximum principle and thus its solutions to Dirichlet problem are unique. It follows that
\begin{equation}\label{eq:relation_for_LAQ_solutions}
\tilde v=e^{i\tilde \varphi}v \text{ in } \Omega.
\end{equation}

Next we relate the solutions of the adjoint equation with coefficients $(A,Q)$ and $(\widetilde A,\widetilde Q)$. We first note that the formal Hermitian adjoint of $L_{A,Q}$ is
\begin{equation}\label{eq:adjoint_operator}
L^*_{A,Q}=L_{\overline A, \overline Q}.
\end{equation}
Now,  if $v_0$ solves $L^*_{A,Q}v_0=0$ with $v_0=f$ on $\p \Omega$, then
\begin{align*}
L^*_{\widetilde A,\widetilde Q}\,e^{i\overline{\tilde\varphi}}v_0
&= L_{\overline{\widetilde A},\overline{\widetilde Q}}\,e^{i\overline{\tilde
\varphi}}v_0
 = L_{\overline{A-\nabla\widetilde\varphi},\overline Q}\,e^{i\overline{
 \tilde\varphi}}v_0= L_{\overline{A}-\nabla\overline{\tilde\varphi}, \overline Q}\,e^{i\overline{\tilde
 \varphi}}v_0.
\end{align*}
We use the identity \eqref{eq:gauge_trans} with $-\overline{\tilde\varphi}$, $\overline A$, $\overline Q$ in place of $\tilde\varphi$,  $A$ and $Q$ respectively. With this, the above equals
\[
L_{\overline{A}-\nabla\overline{\tilde\varphi}, \overline Q}\,e^{i\overline{\tilde
 \varphi}}v_0=e^{i\overline {\tilde\varphi}}L_{\overline{A}, \overline Q}v_0=e^{i\overline{\tilde \varphi}}L^*_{A,Q}v_0=0.
\]


Now, $\tilde v_0$ also solves
\[
L^*_{\wt A, \wt Q}\tilde v_0=0
\]
with $\tilde v_0=f$ on $\p\Omega$. Since $\tilde\varphi=0$ on $\p \Omega$, we have that $\tilde v_0$ and $e^{i\overline{\tilde\varphi}}v$ satisfy the same equation $L^*_{\wt A, \wt Q}u=0$. Moreover, by the second condition in \eqref{eq:assumption_for_coefficients},
we have that $L^*_{\wt A, \wt Q}u=0$ satisfy the maximum principle and thus its solutions to Dirichlet problem are unique. It follows that\begin{equation}\label{eq: relation for solutions of L*_{A,Q}}
  \tilde v_0=e^{i\overline{\tilde\varphi}}v_0 \text{ in } \Omega  
\end{equation}
as well. 
Thus, by \eqref{eq:relation_for_LAQ_solutions}, we have that 
\[
\widetilde L\widetilde V=0 \ \text{ if and only if } \ \widetilde V=e^{i\tilde\varphi}V, \text{ where } LV=0.
\]
Moreover
\[
\widetilde L^*\widetilde V_0=0 \ \text{ if and only if } \ \widetilde V_0=\frac{\sigma+q u_0}{\tilde \sigma+q\tilde u_0}e^{i\overline{\tilde\varphi}}V_0, \text{ where } L^*V_0=0.
\]
We also note that since $\tilde\varphi$ is purely imaginary, we have 
\[
i\overline{\tilde\varphi}=-i\tilde \varphi.
\]
Therefore \eqref{eq: relation for solutions of L*_{A,Q}} becomes 
\begin{equation}\label{relation for v_0 for L*}
    \tilde v_0=e^{-i{\tilde\varphi}}v_0 \text{ in } \Omega
\end{equation}
We collect the above computations in the following.
\begin{lemma}
Let $L$ and $\widetilde L$ be the linearized operators in \ref{eq:linearized_op} corresponding to coefficients $(\sigma,u_0)$ and $(\tilde \sigma,\tilde u_0)$ respectively. Then  $\tilde V$ solves 
\begin{equation}\label{eq:non-linear_well_posedness_eps_dependent3} 
\begin{cases}
\begin{array}{ll}
     \widetilde L \widetilde V=0 & in\ \Omega \\
       \widetilde V=f & on\ \partial \Omega 
\end{array} 
\end{cases}
\end{equation}
if and only if 
\[
\widetilde V=e^{i\tilde\varphi} V
\]
where V solves
\begin{equation}\label{eq:non-linear_well_posedness_eps_dependent4} 
\begin{cases}
\begin{array}{ll}
     L V=0 & in\ \Omega \\
       V=f & on\ \partial \Omega. 
\end{array} 
\end{cases}
\end{equation}

Also, $\widetilde V_0$ solves 
\begin{equation}\label{eq:non-linear_well_posedness_eps_dependent5} 
\begin{cases}
\begin{array}{ll}
     \widetilde L^* \widetilde V_0=0 & in\ \Omega \\
       \widetilde V_0=f & on\ \partial \Omega 
\end{array} 
\end{cases}
\end{equation}
if and only if 
\[
\widetilde V_0=\frac{\sigma+q u_0}{\tilde \sigma+q\tilde u_0}e^{-i\tilde\varphi}V_0
\]
where $V_0$ solves
\begin{equation}\label{eq:non-linear_well_posedness_eps_dependent6} 
\begin{cases}
\begin{array}{ll}
     L^* V_0=0 & in\ \Omega \\
       V_0=f & on\ \partial \Omega.
\end{array} 
\end{cases}
\end{equation}
\end{lemma}

\section{Boundary determination}\label{boundary_determination}
Let us assume $\Lambda_{\sigma,F}=\Lambda_{\tilde \sigma,\widetilde F}$. By writing the corresponding linearized equations in the magnetic Shr\"odinger form \eqref{eq:maganetic_schrodinger}, we have by \cite[Theorem C]{nakamura1995global}, that 
\[
Q=\widetilde Q 
\]
to infinite order on $\p \Omega$, and by 
\cite[Theorem D]{nakamura1995global} that 
\begin{equation}\label{eq:A_on_bndr}
A=\widetilde A
\end{equation}
to zeroth order on $\p \Omega$. By the DN map, we know 
\[
u_0=\tilde u_0
\]
to first order on the boundary.  
We will next recover $\sigma$ to first order from the latter two equations using the following. 
\begin{remark}\label{rem:Runge}
Let us consider replacing $u_0$ by 
\[
\check {u}_0 = u_0 + \eps V + \mathcal{O}_{C^{2,\alpha}(\Omega)}(\eps^2),
\]
for $\eps > 0$ small, where $V$ solves the linearized equation $LV = 0$ of \eqref{eq:quasilin_equation}. Such an expansion holds by the well-posedness result in Theorem \ref{Well-posedness}.

Let $x_0 \in \partial \Omega$. We first extend the coefficients of the linearization $L$ of \eqref{eq:quasilin_equation} at $u_0$ to a slightly larger domain $\widetilde \Omega$ with $\Omega \subset \subset \widetilde \Omega$. By \cite[Theorem I.5.4.1]{bers1964partial}, there exists a solution $V_{\text{loc}}$ to $LV_{\text{loc}} = 0$ in a ball $B \subset \widetilde \Omega$ with prescribed value and gradient at $x_0$. By the Runge approximation property (see e.g. \cite[Proposition A.6]{LLS2020poisson}), there exists a solution $V$ to $LV = 0$ in $\Omega$ that is arbitrarily close to $V_{\text{loc}}$ on $B$ in the $C^1$ norm. By this construction, together with the linearity of $L$, we may find a solution to $LV = 0$ in $\Omega$ such that
\begin{equation}\label{eq:choosing_V}
V(x_0) = 0, \qquad \nabla V(x_0) = V_0,
\end{equation}
where $V_0$ is arbitrarily prescribed.
\end{remark}
Now, we have that 
\[
A = -\frac{i}{2}\frac{\nabla(\sigma + q u_0) + q \nabla u_0}{\sigma + q u_0}.
\]
This can be rewritten as
\[
\nabla \sigma - 2i \sigma A = 2iq u_0 A - 2q \nabla u_0 - u_0 \nabla q.
\]
Similarly, for $\tilde \sigma$ and $\widetilde A$,
\[
\nabla \tilde{\sigma} - 2i \tilde{\sigma} \widetilde{A}
= 2iq u_0 \widetilde{A} - 2q \nabla u_0 - u_0 \nabla q.
\]
Subtracting the two equations and using $A = \widetilde{A}$, we obtain for $\hat{\sigma} := \sigma - \tilde{\sigma}$,
\begin{equation}\label{eq:without_hat}
\nabla \hat{\sigma} - 2i \hat{\sigma} A = 0.
\end{equation}
Here $A$ depends on $u_0$, and the equation holds for any $u_0$ such that $u_0|_{\partial \Omega}$ is in the domain of $\Lambda_{\sigma,F}=\Lambda_{\tilde \sigma,\widetilde F}$.

Let $x_0 \in \partial \Omega$. Next, we replace $u_0$ by $\check u_0 = u_0 + \varepsilon V + \mathcal{O}_{C^{2,\alpha}(\Omega)}(\varepsilon^2)$, where $V$ satisfies $V(x_0)=0$, $\nabla V(x_0) = V_0$ (with $V_0$ arbitrary), and $LV=0$. Such a $V$ exists by the construction in Remark \ref{rem:Runge}. Then the corresponding vector field $\check A$
\[
\check{A}
= -\frac{i}{2}
\frac{\nabla \sigma + 2q \nabla (u_0 + \eps V + \mathcal{O}_{C^{2,\alpha}(\Omega)}(\eps^2)) + (u_0 + \varepsilon V)\nabla q}
{\sigma + q(u_0 + \eps V + \mathcal{O}_{C^{2,\alpha}(\Omega)}(\eps^2))}.
\]
satisfies, at $x_0$,
\[
\check A(x_0) = A(x_0) - \frac{i\varepsilon q(x_0)}{D_0} \nabla V(x_0) + \mathcal{O}(\varepsilon^2),
\]
where $D_0 = \sigma(x_0) + q(x_0)u_0(x_0)$. Since $q \neq 0$, we may choose $\nabla V(x_0)$ so that $\check A(x_0) \neq A(x_0)$ for sufficiently small $\varepsilon$.

Applying the algebraic identity \eqref{eq:without_hat} to $\check A$ reads
\begin{equation}\label{eq:with_hat}
\nabla \hat{\sigma} - 2i \hat{\sigma} \check A = 0.
\end{equation}
Evaluating \eqref{eq:without_hat} and \eqref{eq:with_hat} at $x_0$ and subtracting gives
\[
-2i (A(x_0) - \check A(x_0)) \hat\sigma(x_0) = 0.
\]
Since $A(x_0) \neq \check A(x_0)$, we conclude $\hat\sigma(x_0) = 0$. Substituting back yields $\nabla \hat\sigma(x_0) = 0$. As $x_0 \in \partial \Omega$ was arbitrary, we obtain
\[
\sigma = \tilde{\sigma}, \qquad \nabla\sigma = \nabla\tilde{\sigma} \quad \text{on } \partial\Omega.
\]

We also remark that since $\sigma$ is now known up to first order on $\partial \Omega$, and the Cauchy data for $u_0$ is determined by the boundary value $f_0$ together with the Dirichlet-to-Neumann map, it follows that the Cauchy data for $\sigma + q u_0$ is also known on $\partial \Omega$.

\section{Recovering \texorpdfstring{$\widetilde \varphi$}{phi~} from second linearization}\label{sec:recovering_phi}
We continue the proof of the main theorem by showing that the gauge invariance $\nabla \tilde \varphi$ appearing in the recovery of the coefficients of the first linearized equation can be eliminated considering the second linearization. As mentioned, although we have already obtained \eqref{A-equation} and \eqref{Q-equation}, these identities alone are insufficient due to \eqref{eq:A_tildeA}.

To eliminate this gauge invariance, we analyze the second integral identity
\begin{equation}\label{eq:2nd_lin_integral_id_substracted*}
0 = \int_\Omega V_0 \big( \nabla \cdot (q V_1 \nabla V_2) + \nabla \cdot (q V_2 \nabla V_1) \big)
- \int_\Omega \widetilde V_0 \big( \nabla \cdot (q \widetilde V_1 \nabla \widetilde V_2) + \nabla \cdot (q \widetilde V_2 \nabla \widetilde V_1) \big),
\end{equation}
which follows from \eqref{eq:2nd_lin_integral_id_substracted} together with the boundary determination results in Section \ref{boundary_determination}. Here
\[
L^* V_0 = 0 \quad \text{and} \quad \widetilde L^* \widetilde V_0 = 0.
\]
Note that if $v_0$ solves $L_{A,Q}^* v_0 = 0$, then $(\sigma + q u_0)^{-1} v_0$ solves $L^* V_0 = 0$, and analogously for $\widetilde V_0$. Additionally, $L V_j = 0$ and $\widetilde L \widetilde V_j = 0$ for $j = 1,2$.

We begin by obtaining an equation relating
\[
\Theta := \sigma + q u_0 \quad \text{and} \quad \widetilde \Theta := \tilde \sigma + q \tilde u_0
\]
to $\tilde \varphi$. To this end, observe that \eqref{eq:2nd_lin_integral_id_substracted*} can be rewritten as
\begin{equation}
0 = \int_\Omega V_0 \, \nabla \cdot (q \nabla (V_1 V_2)) - \int_\Omega \widetilde V_0 \, \nabla \cdot (q \nabla (\widetilde V_1 \widetilde V_2)).
\end{equation}
Define
\[
A(x) = e^{2i \tilde \varphi(x)} \quad \text{and} \quad B(x) = \frac{\Theta(x)}{\widetilde \Theta(x)} e^{-i \tilde \varphi(x)}.
\]
With this notation, Lemma \ref{eq:non-linear_well_posedness_eps_dependent3} together with integration by parts yields
\[
\int_\Omega V_1 V_2 \left[ \nabla \cdot (q \nabla V_0) - A \, \nabla \cdot (q \nabla (B V_0)) \right] = 0.
\]
Keeping $V_0$ fixed and choosing $V_1$, $V_2$ as CGO solutions of $L_{A,Q} V = 0$, the density of products of CGOs implies
\begin{equation}\label{eq:AB_eq}
\nabla \cdot (q \nabla V_0) - A \, \nabla \cdot (q \nabla (B V_0)) = 0.
\end{equation}
For $n \geq 3$, this density result is proved in \cite[Section 4]{nakamura1995global}. The statement is also well known in dimension $2$ and can be inferred from \cite{GT2011}, but for completeness, we have included a short proof in the Appendix (see Lemma \ref{lem:density_CGOs_2D}). 
Here, $V_0$ may be taken as any solution of $L^* V_0 = 0$.

Since $L^*$ has no zeroth order term, we may first choose $V_0\equiv 1$. This yields (after dividing by $A$)
\[
\nabla\cdot (q\nabla B)=0.
\]
Since $B=1$ on $\partial \Omega$, uniqueness for the Dirichlet problem yields
\[
B(x)=1 \text{ in } \Omega. 
\]

Next, using $B\equiv 1$, Equation \eqref{eq:AB_eq} becomes
\[
(1-A)\nabla\cdot (q\nabla V_0)=0.
\]
Thus, if we can choose $V_0$ such that the condition $\nabla\cdot (q\nabla V_0)\neq 0$ holds at a given point, then we have that $A=1$, equivalently $\tilde\varphi=0$, at that point. We show next that specific CGO solutions to $L^*V_0=0$ are sufficient for this purpose.

\subsection{CGO solutions for $L^*V_0=0$}
We show that a function $V_0$ satisfying

\[
\nabla\cdot (q\nabla V_0) \neq 0 
\]
can be obtained by taking $V_0$ to be a CGO solution under the condition
\[
\nabla(\sigma/q) \neq 0, \quad \forall x \in \overline{\Omega}. 
\]
CGOs for the magnetic Schr\"odinger equation were constructed in \cite[Section 3]{nakamura1995global}. We note that the CGO construction there also holds verbatim in dimension $2$. These are solutions to
\begin{equation}\label{eq:magnetic_schrodinger}
NU := -\sum_{j=1}^{n} (\partial_{x_j} + i B_j)^2 U + K U = 0,
\end{equation}
of the form
\[
U = e^{\zeta \cdot x + \phi_\zeta} (1 + r_\zeta),
\]
where $\zeta \in \mathbb C^n$ satisfies
\[
\zeta \cdot \zeta = 0 
\]
and $\phi_\zeta$ satisfies
\[
\zeta \cdot \nabla \phi_\zeta = -i \zeta \cdot B.
\]

Since we assume that our coefficients $\sigma,q,F$ and $f_0$ are $C^\infty$ smooth, a standard argument differentiating the equation $N(e^{\zeta\cdot x+\phi_\zeta}(1+r_\zeta))=0$, using Carleman estimates and Sobolev embedding, shows that
\[
r_\zeta=O_{C^2(\Omega)}(1)
\]
as $\abs{\zeta}\to \infty$. 

We first write the equation $L^*V_0=0$ for $V_0$ in the magnetic Shr\"odinger form used in \cite{nakamura1995global}. From \eqref{eq:Lstar}, we have
\[
L^*V_0=\nabla\cdot (\sigma+qu_0)\nabla V_0-q\nabla u_0\cdot \nabla V_0=\nabla\sigma\cdot \nabla V+u_0\nabla q\cdot \nabla V_0+(\sigma+qu_0)\Delta V_0=0. 
\]
This is equivalent to $V_0$ satifying
\[
-\Delta V_0-\frac{\nabla \sigma+u_0\nabla q}{\sigma+qu_0}\cdot \nabla V_0=-\Delta V_0+Z\cdot \nabla V_0=0,
\]
where 
\begin{equation}\label{eq:Z_formula}
Z=-\frac{\nabla \sigma+u_0\nabla q}{\sigma+qu_0}.
\end{equation}
Thus, we obtain that $V_0$ is a solution to (see Section \ref{sec:sol_to_first_lin})
\[
NV_0=0,
\]
where
\[
B=\frac{iZ}{2} \text{ and } K=-\frac 12 \nabla \cdot Z+\frac 14 Z\cdot Z. 
\]

We have 
\[
    \nabla\cdot (q\nabla V_0)=\nabla\cdot (q\nabla (e^{\zeta\cdot x+\phi_\zeta}(1+r_\zeta)))=\nabla\cdot (q\nabla (e^{\zeta\cdot x+\phi_\zeta}))+e^{\zeta\cdot x+\phi_\zeta}O_{C^0}(1).
\]
We denote $\phi=\phi_\zeta$ and compute
\begin{multline}
\nabla\cdot (q\nabla (e^{\zeta\cdot x+\phi}))=\nabla\cdot\big( q(\zeta+\nabla \phi)e^{\zeta\cdot x+\phi})\big)=q\nabla\cdot \big((\zeta+\nabla \phi)e^{\zeta\cdot x+\phi}\big)+\nabla q \cdot (\zeta+\nabla \phi)e^{\zeta\cdot x+\phi} \\
=q\Delta \phi e^{\zeta\cdot x+\phi}+q (\zeta+\nabla \phi)\cdot (\zeta+\nabla \phi) e^{\zeta\cdot x+\phi}+\nabla q \cdot (\zeta+\nabla \phi)e^{\zeta\cdot x+\phi}.
\end{multline}
We have
\[
(\zeta+\nabla \phi)\cdot (\zeta+\nabla \phi)=2\zeta \cdot \nabla \phi+\nabla\phi\cdot \nabla\phi
\]
since $\zeta \cdot \zeta=0$. It follows that 
\[
\nabla\cdot (q\nabla V_0)\neq 0 
\]
if and only if 
\[
q\Delta \phi +2q \zeta \cdot \nabla \phi+\nabla\phi\cdot \nabla\phi +\nabla q \cdot (\zeta+\nabla \phi)+O_{C^0}(1)\neq 0.
\]
We divide this equation by $\abs{\zeta}$ and take $\abs{\zeta}\to \infty$ to obtain
\begin{equation}\label{eq:CGO_condition_to_satisfy}
(2q \nabla \phi+\nabla q) \cdot \hat\zeta\neq 0,
\end{equation}
where $\hat \zeta=\zeta/\abs{\zeta}$. 

Next we recall that $\phi$ satisfies
\[
\zeta\cdot \nabla \phi=-i\zeta\cdot B
\]
and
\[
B=\frac{iZ}{2}.
\]
Using these in \eqref{eq:CGO_condition_to_satisfy} yields the condition
\[
(q Z+\nabla q) \cdot \hat\zeta\neq 0.
\]
Thus we are done if $q Z+\nabla q\neq 0$. By the formula \eqref{eq:Z_formula} for $Z$, this is equivalent to 
\[
(\sigma+qu_0)\nabla q -q(\nabla \sigma +u_0\nabla q)\neq 0,
\]
but this is exactly the structure condition
\begin{equation}\label{eq:strc_condition_proof}
\nabla(\sigma/q) \neq 0.
\end{equation}

It follows that if $x_0\in \Omega$ is fixed, we may choose $\zeta=\zeta_{x_0}$ and corresponding CGO $V_0$ such that
\[
\nabla\cdot (q \nabla V_0)(x_0)\neq 0.
\]
Repeating the argument for all $x_0\in \Omega$ shows that
\[
A\equiv 1,
\]
or equivalently 
\begin{equation}\label{eq:varphi_recovery}
\tilde \varphi\equiv 0
\end{equation}
since by Section \ref{boundary_determination}, $\tilde\varphi$ vanishes on the boundary.

\begin{remark}
A potential alternative approach to constructing solutions $V_0$ satisfying $L^*V_0=0$ together with the condition $\nabla\cdot (q\nabla V_0)=1$ holding at a given point is via Runge approximation. Specifically, one expects that the values of $\nabla^2 V_0$, $\nabla V_0$, and $V_0$ itself can be prescribed arbitrarily at a point, subject only to the constraint that $L^*V_0=0$ holds at that point.

Setting $a=\Delta V_0(x_0)$ and $\eta=\nabla V_0(x_0)$ at a point $x_0\in \Omega$, the desired conditions 
\[
\nabla\cdot (q\nabla V_0)(x_0)=1 \quad\text{and}\quad L^*V_0(x_0)=0
\]
reduce to the linear system
\begin{align*}
\nabla q\cdot \eta + q a &= 1, \\
\nabla\sigma\cdot \eta + \sigma a &= -u_0.
\end{align*}
For $q\neq 0$, this system admits a solution $(a,\eta)$ precisely when
\[
\nabla(\sigma/q) \neq 0, \quad \forall x \in \overline{\Omega},
\]
which coincides with condition \eqref{eq:strc_condition_proof} we obtained via the CGO construction.

While we were unable to locate a direct reference for the required Runge approximation result, we anticipate that such could be proven by following the proof of \cite[Proposition A.5]{LLS2020poisson}. For the lack of reference, we employed CGOs instead.
\end{remark}

\section{Concluding the proof}
In this section we conclude the proof of Theorem \ref{thm:Main theorem}. Let us recall what we have recovered so far. From \eqref{eq:A_tildeA} we have
\begin{equation*} 
A - \widetilde A = \nabla \tilde\varphi \quad \text{in } \Omega,
\end{equation*}
where
\begin{align*} 
A &= -\frac{i}{2}\frac{\nabla(\sigma + q u_0) + q \nabla u_0}{\sigma + q u_0}
   = -\frac{i}{2}\frac{\nabla\Theta + q \nabla u_0}{\Theta},\\[6pt]
\widetilde A &= -\frac{i}{2}\frac{\nabla(\tilde\sigma + q \tilde u_0) + q \nabla \tilde u_0}{\tilde\sigma + q \tilde u_0}
   = -\frac{i}{2}\frac{\nabla\widetilde\Theta + q \nabla \tilde u_0}{\widetilde\Theta}.
\end{align*}
In Section \ref{sec:recovering_phi} we showed $\tilde\varphi = 0$. Thus
\begin{equation}\label{eq:A=tildeA}
A = \widetilde A \quad \text{in } \Omega.
\end{equation}
Taking the divergence of \eqref{eq:A=tildeA} yields
\begin{multline}\label{eq:Term A}
\frac{\Delta(\Theta-\widetilde\Theta)}{\Theta} + q \frac{\Delta (u_0-\tilde u_0)}{\Theta} + \Delta \widetilde\Theta \left(\frac{1}{\Theta}-\frac{1}{\widetilde\Theta}\right) + q\Delta \tilde u_0 \left(\frac{1}{\Theta}-\frac{1}{\widetilde\Theta}\right) \\ + \underbrace{R_1(\Theta, \nabla \Theta,u_0, \nabla u_0) - R_1(\widetilde \Theta, \nabla \widetilde \Theta,\tilde u_0, \nabla \tilde u_0)}_{\hat{R}_1} = 0,
\end{multline}
where $R_1 = R_1(s,\eta,r,\mu)$, with $s,r\in \mathbb{R}$ and $\eta,\mu\in \mathbb{R}^n$, is a rational function in $s$ and a polynomial in its other variables.

We also obtained $Q = \widetilde Q$, see \eqref{eq:Q_recovery}, where
\begin{align*} 
Q &= \frac{\Delta(\sigma+qu_0)-q\Delta u_0}{2(\sigma + qu_0)} 
    - \frac{\nabla q \cdot \nabla u_0}{2(\sigma+qu_0)} 
    + \frac{q^2 |\nabla u_0|^2 - |\nabla (\sigma+qu_0)|^2}{4(\sigma+qu_0)^2},\\[6pt]
\widetilde Q &= \frac{\Delta(\tilde\sigma+q\tilde u_0)-q\Delta \tilde u_0}{2(\tilde \sigma + q\tilde u_0)} 
    - \frac{\nabla q \cdot \nabla \tilde u_0}{2(\tilde\sigma+q\tilde u_0)} 
    + \frac{q^2 |\nabla \tilde u_0|^2 - |\nabla (\tilde\sigma+q \tilde u_0)|^2}{4(\tilde\sigma+q\tilde u_0)^2}.
\end{align*}
Thus we have
\begin{multline}\label{eq: Term Q}
\frac{\Delta(\Theta-\widetilde\Theta)}{2\Theta} - q \frac{\Delta (u_0-\tilde u_0)}{2\Theta} 
+ \frac{\Delta \widetilde\Theta}{2}\left(\frac{1}{\Theta}-\frac{1}{\widetilde \Theta}\right) 
- \frac{q \Delta\tilde u_0}{2}\left(\frac{1}{\Theta}-\frac{1}{\widetilde \Theta}\right) \\ 
+ \underbrace{R_2(\Theta, \nabla \Theta,u_0, \nabla u_0) - R_2(\widetilde \Theta, \nabla \widetilde \Theta,\tilde u_0, \nabla \tilde u_0)}_{\hat{R}_2} = 0,
\end{multline}
where $R_2 = R_2(s,\eta,r,\mu)$ is a rational function in $s$ and a polynomial in its other variables.

Let us denote $\hat{\Theta} := \Theta - \widetilde\Theta$ and $\hat{u}_0 := u_0 - \tilde u_0$. Then \eqref{eq:Term A} and \eqref{eq: Term Q} can be rewritten as
\begin{align} 
\frac{\Delta\hat{\Theta}}{\Theta} + q \frac{\Delta \hat{u}_0}{\Theta} 
+ \underbrace{\Delta \widetilde\Theta \left(\frac{1}{\Theta}-\frac{1}{\widetilde\Theta}\right) 
+ q\Delta \tilde u_0 \left(\frac{1}{\Theta}-\frac{1}{\widetilde \Theta}\right)}_{L_1(1/\Theta)-L_1(1/\widetilde\Theta)=\hat{L}_1} 
+ \hat{R}_1 &= 0 \tag{a}\\[6pt]
\frac{\Delta\hat{\Theta}}{2\Theta} - q \frac{\Delta \hat{u}_0}{2\Theta} 
+ \underbrace{\frac{\Delta \widetilde\Theta}{2}\left(\frac{1}{\Theta}-\frac{1}{\widetilde \Theta}\right) 
- \frac{q \Delta\tilde u_0}{2}\left(\frac{1}{\Theta}-\frac{1}{\widetilde \Theta}\right)}_{L_2(1/\Theta)-L_2(1/\widetilde\Theta)=\hat{L}_2} 
+ \hat{R}_2 &= 0. \tag{b}
\end{align}
Since $\widetilde \Theta, \tilde u_0 \in C^{\infty}(\overline{\Omega})$, we have $\Delta \widetilde\Theta, \Delta \tilde u_0 \in C(\overline{\Omega})$. Consequently, $L_a$ for $a=1,2$ are Lipschitz.

Equations (a) and (b) yield the following coupled elliptic system for $\hat{\Theta}$ and $\hat{u}_0$:
\begin{equation}\label{eq:M_matrix_eq}
M\begin{pmatrix} \Delta \hat{\Theta}\\ \Delta\hat{u}_0 \end{pmatrix}
:= 
\begin{pmatrix}
\frac{1}{\Theta} & \frac{q}{\Theta} \\[4pt]
\frac{1}{2\Theta} & -\frac{q}{2\Theta}
\end{pmatrix}
\begin{pmatrix} \Delta \hat{\Theta}\\ \Delta\hat{u}_0 \end{pmatrix}
=
\begin{pmatrix}
\hat{R}_1 + \hat{L}_1 \\[4pt]
\hat{R}_2 + \hat{L}_2
\end{pmatrix},
\end{equation}
where $\hat{R}_a$ and $\hat{L}_a$ contain at most one derivative of $\Theta$, $\widetilde \Theta$ and $u_0$, $\tilde u_0$. 
Since
\[
\det(M) = -\frac{q}{\Theta^2} \neq 0
\]
by the assumptions \eqref{eq:well-posedness_conditions} and $q \neq 0$, the matrix $M$ is invertible. Therefore, we obtain the elliptic system
\[
\begin{pmatrix} \Delta \hat{\Theta}\\ \Delta\hat{u}_0 \end{pmatrix}
= M^{-1}
\begin{pmatrix} \hat{R}_1 + \hat{L}_1 \\ \hat{R}_2 + \hat{L}_2 \end{pmatrix}
=
\begin{pmatrix}
\frac{\Theta}{2} & \Theta \\[4pt]
\frac{\Theta}{2q} & -\frac{\Theta}{q}
\end{pmatrix}
\begin{pmatrix} \hat{R}_1 + \hat{L}_1 \\ \hat{R}_2 + \hat{L}_2 \end{pmatrix},
\]
which is diagonal in the leading order. This may also be written as
\begin{equation*}
\begin{pmatrix} \Delta \hat{\Theta}\\ \Delta\hat{u}_0 \end{pmatrix}
= L(\Theta, \nabla\Theta, u_0, \nabla u_0) - L(\widetilde \Theta, \nabla\widetilde\Theta, \tilde u_0, \nabla \tilde u_0),
\end{equation*}
where $L$ is a vector-valued Lipschitz function in all its variables.

Since $\hat \Theta$ and $\hat{u}_0$ vanish to first order on the boundary $\partial \Omega$ by Section \ref{boundary_determination}, the unique continuation principle for systems (see e.g. \cite[Theorem B.1]{KMU2011}) implies
\[
(\hat{\Theta}, \hat{u}_0) \equiv (0,0).
\]
Hence,
\[
u_0 = \tilde{u}_0 \ \  \text{ in } \Omega 
\]
Using the identity
\[
\tilde{\sigma} + q\tilde{u}_0 = \widetilde{\Theta} = \Theta = \sigma + q u_0,
\]
it then follows that
\[
\tilde{\sigma} = \sigma \ \  \text{ in } \Omega.
\]
Finally, substituting these identities into equation \eqref{eq:quasilin_equation}, we obtain
\[
\widetilde{F}
= \nabla \cdot \big((\tilde{\sigma}+q\tilde{u}_0)\nabla \tilde{u}_0\big)
= \nabla \cdot \big((\sigma+q u_0)\nabla u_0\big)
= F
\]
holding in $\Omega$. 
This completes the proof of Theorem \ref{thm:Main theorem}.
\hfill $\square$

\section*{Appendix A: Density of CGOs in dimension 2}
In this appendix we recall the density of products of complex geometric optics (CGO) solutions for the magnetic Schr\"odinger equation in dimension $2$.

\begin{lemma}\label{lem:density_CGOs_2D}
Let $f\in C_0(\Omega)$. If
\begin{equation}\label{eq:density_2D}
\int_\Omega f \, \omega_1 \omega_2 = 0
\end{equation}
for all $\omega_1$ and $\omega_2$ satisfying
\begin{equation*}
L_{C,E}\omega := -\sum_{j=1}^{2} (\partial_{x_j} + i C_j)^2 \omega + E\omega = 0,
\end{equation*}
then $f \equiv 0$.
\end{lemma}

\begin{proof}
In dimension $2$, there exist CGO solutions of the form
\begin{equation*}
\omega_1 = F_A^{-1} e^{\Phi/h} (1 + r_h), \qquad 
\omega_2 = F_{\overline A} e^{-\overline\Phi/h} (1 + \widetilde r_h),
\end{equation*}
where $\Phi$ is holomorphic and Morse, and $r_h, \widetilde r_h = O_{L^2}(h^{1/2+\varepsilon})$ as $h \to 0$, and the functions $F_A$ and $F_{\overline A}$ are nonvanishing. For details, see \cite[Section 5]{liimatainen2025inverse}, which builds on \cite{GT2011}.

Let $z_0\in \Omega\subset \C$. By choosing
\[
\Phi(z) = (z-z_0)^2
\]
and substituting $\omega_1$ and $\omega_2$ into \eqref{eq:density_2D}, an application of the stationary phase method, using $f|_{\p\Omega}=0$, yields $f(z_0)=0$. Varying $z_0$ shows that $f \equiv 0$ in $\Omega$.
\end{proof}

\bibliography{ref} 
				
				\bibliographystyle{alpha}
\end{document}